\documentclass[a4paper, 12pt]{article}

\usepackage[sort&compress]{natbib}
\bibpunct{(}{)}{;}{a}{}{,} 

\usepackage{amsthm, amsmath, amssymb, mathrsfs, multirow, url, subfigure}
\usepackage{graphicx} 
\usepackage{ifthen} 
\usepackage{amsfonts}
\usepackage[usenames,dvipsnames]{color}
\usepackage{fullpage}

\numberwithin{equation}{section} 

\theoremstyle{plain} 
\newtheorem{thm}{Theorem}
\newtheorem{prop}{Proposition}

\theoremstyle{definition}
\newtheorem{defn}{Definition}

\theoremstyle{remark} 

\newtheorem{remark}{Remark}
\newtheorem*{astep}{A-step}
\newtheorem*{pstep}{P-step}
\newtheorem*{cstep}{C-step}

\newcommand{\prob}{\mathsf{P}}

\newcommand{\RR}{\mathbb{R}}

\newcommand{\XX}{\mathbb{X}}

\newcommand{\UU}{\mathbb{U}}
\newcommand{\VV}{\mathbb{V}}

\newcommand{\Xbar}{\bar{X}}
\newcommand{\xbar}{\bar{x}}
\newcommand{\Ybar}{\bar{Y}}

\newcommand{\Ubar}{\bar{U}}

\renewcommand{\S}{\mathcal{S}}

\newcommand{\Sbar}{\bar{\mathcal{S}}}

\newcommand{\del}{\partial}

\renewcommand{\phi}{\varphi}

\newcommand{\bel}{\mathsf{bel}}
\newcommand{\pl}{\mathsf{pl}}
\newcommand{\mbel}{\mathsf{mbel}}
\newcommand{\mpl}{\mathsf{mpl}}

\newcommand{\unif}{\mathsf{Unif}}
\newcommand{\nm}{\mathsf{N}}

\newcommand{\iid}{\overset{\text{\tiny iid}}{\,\sim\,}}

\title{Marginal inferential models: prior-free probabilistic inference on interest parameters}
\author{
Ryan Martin \\
Department of Mathematics, Statistics, and Computer Science \\
University of Illinois at Chicago \\
\url{rgmartin@uic.edu} \\
\mbox{} \\
Chuanhai Liu \\
Department of Statistics \\
Purdue University \\
\url{chuanhai@purdue.edu}
}
\date{\today}

\begin{document}

\maketitle

\begin{abstract}
The inferential models (IM) framework provides prior-free, frequency-calibrated, posterior probabilistic inference.  The key is the use of random sets to predict unobservable auxiliary variables connected to the observable data and unknown parameters.  When nuisance parameters are present, a marginalization step can reduce the dimension of the auxiliary variable which, in turn, leads to more efficient inference.  For regular problems, exact marginalization can be achieved, and we give conditions for marginal IM validity.  We show that our approach provides exact and efficient marginal inference in several challenging problems, including a many-normal-means problem.  In non-regular problems, we propose a generalized marginalization technique and prove its validity.  Details are given for two benchmark examples, namely, the Behrens--Fisher and gamma mean problems.    

\smallskip

\emph{Keywords and phrases:} Belief; efficiency; nuisance parameter; plausibility; predictive random set; validity.
\end{abstract}

\section{Introduction}
\label{S:intro}

In statistical inference problems, it is often the case that only some component or, more generally, some feature of the parameter $\theta$ is of interest.  For example, in linear regression, with $\theta=(\beta,\sigma^2)$, often only the vector $\beta$ of slope coefficients is of interest, even though the error variance $\sigma^2$ is also unknown.  Here we partition $\theta$ as $\theta = (\psi,\xi)$, where $\psi$ is the parameter of interest and $\xi$ is the nuisance parameter.  The goal is to make valid and efficient inference on $\psi$ in the presence of unknown $\xi$.  

In these nuisance parameter problems, a modification of the classical likelihood framework is called for.  Frequentists often opt for profile likelihood methods \citep[e.g.,][]{cox2006}, where the unknown $\xi$ is replaced by its conditional maximum likelihood estimate $\hat\xi_\psi$.  The effect is that the likelihood function involves only $\psi$, so, under some conditions, point estimates and hypothesis tests with desirable properties can be constructed as usual.  The downside, however, is that no uncertainty in $\xi$ is accounted for when it is fixed at its maximum likelihood estimate.  A Bayesian style alternative is the marginal likelihood approach, which assumes an \emph{a priori} probability distribution for $\xi$.  The marginal likelihood for $\psi$ is obtained by integrating out $\xi$.  This marginal likelihood inference effectively accounts for uncertainty in $\xi$, but difficulties arise from the requirement of a prior distribution for $\xi$.  Indeed, suitable reference priors may not be available or there may be marginalization problems \citep[e.g.,][]{dawid.stone.zidek.1973}.     

For these difficult problems, something beyond standard frequentist and Bayesian approaches may be needed.  There has been a recent effort to develop probabilistic inference without genuine priors; see, for example, generalized fiducial inference \citep{hannig2012, hannig2009}, confidence distributions \citep{xie.singh.strawderman.2011, xie.singh.2012}, and Bayesian inference with default, reference, or data-dependent priors \citep{berger2006, bergerbernardosun2009, fraser2011, fraser.reid.marras.yi.2010}.  The key idea behind Fisher's fiducial argument is to write the sampling model in terms of an auxiliary variable $U$ with known distribution $\prob_U$, and then ``continue to regard'' \citep{dempster1964} $U$ as having distribution $\prob_U$ after data is observed.  \citet{liu.martin.wire} argue that the ``continue to regard'' assumption has an effect similar to a prior; in fact, the generalized fiducial distribution is a Bayesian posterior under a possibly data-dependent prior \citep{hannig2012}.  Therefore, fiducial inference is not prior-free, and generally is not exact for finite samples, which explains the need for asymptotic justification.  This same argument applies to the other fiducial-like methods, including structural inference \citep{fraser1968}, Dempster--Shafer theory \citep{dempster2008, shafer1976}, and those mentioned above.  

The \emph{inferential model} (IM) approach, described in \citet{imbasics, imbasics.c, imcond} is a new alternative that is both prior-free and is provably exact in finite samples; see, also, \citet{mzl2010} and \citet{zl2010}.  The fundamental idea behind the IM approach is that inference on an unknown parameter $\theta$ is equivalent to predicting an unobserved auxiliary variable that has a known distribution.  This view of inference through \emph{prediction} of auxiliary variables differs from fiducial's ``continue to regard,'' and is essential to exact probabilistic inference without priors \citep{sts.discuss.2014}.  The practical consequences of this approach are two-fold.  First, no prior is needed, and yet the inferential output is probabilistic and has a meaningful interpretation after data is observed.  Second, calibration properties of the IM output are determined by certain coverage probabilities of user-defined random sets, and \citet{imbasics} show that constructing valid random sets is quite simple, especially when the auxiliary variable is low-dimensional.  When the auxiliary variable is high-dimensional, constructing efficient random sets can be challenging.  A natural idea is to reduce the dimension of the auxiliary variable as much as possible before prediction.  \citet{imcond} notice that, in many problems, certain functions of the auxiliary variable are actually observed, so it is not necessary to predict the full auxiliary variable.  They propose a general method of dimension reduction based on conditioning.  In marginal inference problems, where only parts of the full parameter are of interest, we propose to reduce the dimension of the auxiliary variable even further.  Such considerations are particularly important for high-dimensional applications, where the quantity of interest typically resides in a lower-dimensional subspace, and an implicit marginalization is required.  Here we develop the IM framework for marginal inference problems based on a second dimension reduction technique.  If the model is regular in the sense of Definition~\ref{def:regular}, then the specific dimension on which to collapse is clear, and a marginal IM obtains.  Sections~\ref{SS:valid} and \ref{SS:optimal} discuss validity and efficiency of marginal IMs in the regular case.  Several examples of challenging marginal inference problems are given in Section~\ref{S:applications}.  

When the model is not regular, and the interest and nuisance parameters cannot easily be separated, a different strategy is needed.  The idea is that non-separability of the interest and nuisance parameters introduces some additional uncertainty, and we handle this by taking larger predictive random sets.  In Section~\ref{S:gmim}, we describe a marginalization strategy for non-regular problems based on uniformly valid predictive random sets.  Details are given for two benchmark examples: the Behrens--Fisher and gamma mean problems.  We conclude with a brief discussion in Section~\ref{S:discuss}.  

\section{Review of IMs}
\label{S:im.review}

\subsection{Three-step construction}
\label{SS:three.step}

To set notation, let $X \in \XX$ be the observable data.  The sampling model $\prob_{X|\theta}$, indexed by a parameter $\theta \in \Theta$, is a probability measure that encodes the joint distribution of $X$.  If, for example, $X=(X_1,\ldots,X_n)$ is a vector of iid observables, then $\XX$ is a product space and $\prob_{X|\theta}$ is an $n$-fold product measure.  

A critical component of the IM framework is an association between observable data $X$ and unknown parameters $\theta$ via unobservable auxiliary variables $U$, subject to the constraint that the induced distribution of $X$, given $\theta$, agrees with the posited sampling model $\prob_{X|\theta}$.  In this paper we express $\prob_{X|\theta}$ as follows: given $\theta \in \Theta$, choose $X$ to satisfy
\begin{equation}
\label{eq:amodel}
p(X,\theta) = a(U,\theta), \quad \text{where} \quad U \sim \prob_U.  
\end{equation}
The mathematical expression of this association is more general than that in \citet{imbasics}.  They consider $p(x,\theta) \equiv x$, which is most easily seen as a data-generation mechanism, or structural equation \citep{hannig2009, fraser1968}.  The case $a(u,\theta) \equiv u$ is the pivotal equation version \citep{dawidstone1982}.  Even more general expressions are possible, e.g., $a(X, \theta, U)=0$, but we stick with that in \eqref{eq:amodel} because, in some cases, it is important that $X$ and $U$ can be separated \citep[][Remark~1]{imcond}.  The point is that any sampling model $\prob_{X|\theta}$ that can be simulated with a computer has an association, and the form \eqref{eq:amodel} is general enough to cover many models \citep{barnard1995}.  However, there may be several associations for a given sampling model.  In fact, the key contribution of this paper boils down to showing how to choose an association that admits valid and efficient marginal inference on interest parameters.  

The key observation driving the IM approach is that uncertainty about $\theta$, given observed data $X=x$, is due to the fact that $U$ is unobservable.  So, the best possible inference corresponds to observing $U$ exactly.  This ``best possible inference'' is unattainable, and the next best thing is to accurately predict the unobservable $U$.  This point is what sets the IM approach apart from fiducial inference and its relatives.  Next is the three-step IM construction \citep{imbasics}, in which the practitioner specifies both the association and the predictive random set for the auxiliary variable.  

\begin{astep}
Associate the observed data $x$ and the unknown parameter $\theta$ with auxiliary variables $u$ to obtain sets of candidate parameter values given by 
\[ \Theta_x(u) = \{\theta: p(x,\theta) = a(u,\theta)\}, \quad u \in \UU. \]
\end{astep}

\begin{pstep}
Predict $U$ with a predictive random set $\S \sim \prob_{\S}$.  This predictive random set serves two purposes.  First, it encodes the additional uncertainty in predicting an unobserved value compared to sampling a new value.  Second, if $\S$ satisfies a certain coverage property (Definition~\ref{def:prs.valid}), then the resulting belief function \eqref{eq:lev} is valid (Definition~\ref{def:validity}).
\end{pstep}

\begin{cstep}
Combine $\Theta_x$ with $\S$ to obtain $\Theta_x(\S)$, an enlarged random set of candidate $\theta$ values, given by 
\[ \Theta_x(\S) = \bigcup_{u \in \S} \Theta_x(u). \]
Suppose that $\Theta_x(\S)$ is non-empty with $\prob_\S$-probability~1.  Then for any assertion $A \subseteq \Theta$, summarize the evidence in $X=x$ supporting the truthfulness of $A$ with the quantities 
\begin{align}
\bel_x(A;\S) & = \prob_{\S}\bigl\{\Theta_x(\S) \subseteq A \bigr\} \label{eq:lev} \\
\pl_x(A;\S) & = 1-\bel_x(A^c;\S), \label{eq:uev}
\end{align}
called the belief and plausibility functions, respectively.  Note that, for example, $\bel_x(\cdot;\S)$ is short-hand for $\bel_x(\cdot;\prob_\S)$; that is, both belief and plausibility are functions of data $x$, assertion $A$, and the distribution $\prob_\S$ of the predictive random set.  
\end{cstep}

If $\Theta_x(\S) = \varnothing$ with positive $\prob_\S$-probability, then some adjustment to the formula \eqref{eq:lev} is needed.  The simplest approach, which is Dempster's rule of conditioning \citep{shafer1976}, is to replace $\prob_\S$ by the conditional distribution of $\S$ given $\Theta_x(\S) \neq \varnothing$.  A more efficient alternative is to use elastic predictive random sets \citep{leafliu2012}, which amounts to stretching $\S$ just enough that $\Theta_x(\S)$ is non-empty; see Section~\ref{SS:stein2}.  



To summarize, the practitioner specifies the association and a valid predictive random set, and the output is the pair of mappings $(\bel_x,\pl_x)$ which are used for inference about $\theta$.  For example, if $\pl_x(A;\S)$ is small, then there is little evidence in the data supporting the claim that $\theta \in A$; see \citet{impval}.  The plausibility function can also be used to construct plausibility regions.  That is, for $\alpha \in (0,1)$, a $100(1-\alpha)\%$ plausibility region is $\{\theta: \pl_x(\theta;\S) > \alpha\}$, with $\pl_x(\theta;\S) = \pl_x(\{\theta\}; \S)$.  If $\S$ is valid, then the plausibility region achieves the nominal frequentist coverage probability; see Theorem~\ref{thm:old.valid}.  

Concerning uniqueness of the practitioner's choices, it follows from \citet{imbasics} that an IM constructed using the ``optimal'' predictive random set (see Section~\ref{SS:optimal}) only depends on the sampling model, not on the choice of association.  Optimal predictive random sets  in general are theoretically and computationally challenging, and work in this direction is ongoing.  However, as we discuss in more detail below, if the auxiliary variable dimension can be reduced, then it is possible to avoid the challenge finding good predictive random sets for high-dimensional auxiliary variables.

\subsection{Properties}
\label{SS:im.properties}

The key to IM validity (Definition~\ref{def:validity}) is a certain calibration property of the predictive random set $\S$ in the P-step above.  

\begin{defn}
\label{def:prs.valid}
Let $\S$ be a predictive random set for $U \sim \prob_U$ in \eqref{eq:amodel}, and define $Q_\S(u) = \prob_\S\{\S \not\ni u\}$.  Then $\S$ is called \emph{valid} if $Q_\S(U)$ is stochastically no larger than $\unif(0,1)$, i.e., $\prob_U\{Q_\S(U) \geq 1-\alpha\} \leq \alpha$ for all $\alpha \in (0,1)$.
\end{defn}

It is relatively easy to construct valid predictive random sets.  Indeed, \citet{imbasics, imbasics.c} provide a general sufficient conditions for validity, one being that $\S$ have a nested support, i.e., for any two possible realizations of $\S$, one is a subset of the other.  Suppose that $U$ is one-dimensional and continuous.  In the case, without loss of generality, assume $\prob_U$ is $\unif(0,1)$.  The ``default'' predictive random set 
\begin{equation}
\label{eq:default.prs}
\S = \{u: |u-0.5| \leq |U-0.5|\}, \quad U \sim \unif(0,1), 
\end{equation}
has nested support and is also valid in the sense of Definition~\ref{def:prs.valid} \citep[][Corollary~1]{imbasics}.  We use this default random set in all of our examples.  

It turns out that validity of the predictive random set is all that is needed for validity of the corresponding IM.  Here IM validity refers to a calibration property of the corresponding belief/plausibility function.  

\begin{defn}
\label{def:validity}
Suppose $X \sim \prob_{X|\theta}$ and let $A \subseteq \Theta$.  Then the IM is \emph{valid for $A$} if the belief function satisfies
\begin{equation}
\label{eq:valid}
\sup_{\theta \not\in A} \prob_{X|\theta}\bigl\{ \bel_X(A;\S) \geq 1-\alpha \bigr\} \leq \alpha, \quad \forall \; \alpha \in (0,1). 
\end{equation}
The IM is called \emph{valid} if it is valid for all $A$, in which case, we can write, for all $A$, 
\begin{equation}
\label{eq:valid.pl}
\sup_{\theta \in A} \prob_{X|\theta}\{\pl_X(A;\S) \leq \alpha\} \leq \alpha, \quad \forall \; \alpha \in (0,1). 
\end{equation}
\end{defn}


\begin{thm}
\label{thm:old.valid}
If $\S$ is a valid predictive random set for $U$, and $\Theta_x(\S) \neq \varnothing$ with $\prob_\S$-probability~1 for all $x$, then the IM is valid, i.e., \eqref{eq:valid} and \eqref{eq:valid.pl} hold for all assertions.  
\end{thm}

\begin{proof}
See the proof of Theorem~2 in \citet{imbasics}.  
\end{proof}

One consequence of Theorem~\ref{thm:old.valid} is that the plausibility region for $\theta$ achieves the nominal coverage probability.  More importantly, validity provides a scale on which the IM belief and plausibility function values can be interpreted.  Note that this calibration property is not asymptotic and does not depend on the particular class of models.  

Validity is desirable, but efficiency is also important.  See \citet{imbasics} and Section~\ref{SS:optimal} below for details about IM efficiency.  Unfortunately, specifying predictive random sets so that the corresponding IM is efficient is more challenging, especially if the auxiliary variable is of relatively high dimension.  It is, therefore, natural to try to reduce the dimension of the auxiliary variable.  \citet{imcond} proposed a conditioning operation that effectively reduces the dimension of the auxiliary variable, and they give examples to demonstrate the efficiency gain; see Section~\ref{SS:illustration} below.  The focus of the present paper is the case where nuisance parameters are present, and our main contribution is to demonstrate that such problems often admit a further dimension reduction that yields valid marginal inference without loss of efficiency.  

Throughout the paper, we generally will not distinguish between the concepts of efficiency and auxiliary variable dimension reduction.  The reason is that if the dimension can be successfully reduced, then the challenge of constructing a good predictive random set for a relatively high-dimensional auxiliary variable can be avoided; simple choices, such as the default predictive random set \eqref{eq:default.prs} will suffice for efficient marginal inference.  Section~\ref{SS:optimal} makes this connection rigorous.

\subsection{Illustration: normal mean problem}
\label{SS:illustration}

Let $X_1,\ldots,X_n$ be independent $\nm(\theta,\sigma^2)$ observations, with $\sigma$ known but $\theta$ unknown.  In vector notation, $X = \theta 1_n + \sigma U$, where $U \sim \nm_n(0,I_n)$, $1_n$ is an $n$-vector of unity, and $I_n$ is the $n \times n$ identity matrix.  It would appear that the IM approach requires that we predict the entire unobservable $n$-vector $U$.  However, certain functions of $U$ are observed, making it unnecessary to predict the full vector.  Rewrite the association as 
\begin{subequations}
\label{eq:normal.mean}
\begin{align}
\Xbar & = \theta + \sigma n^{-1/2} \Ubar \label{eq:normal.mean.a} \\
(X_i-\Xbar)/\sigma & = U_i - \Ubar, \quad i=1,\ldots,n. \label{eq:normal.mean.b}
\end{align}
\end{subequations}
Since \eqref{eq:normal.mean.b} does not depend on $\theta$, the residuals $U_i-\Ubar$ are observed.  Motivated by this, \citet{imcond} propose to take \eqref{eq:normal.mean.a} as their (conditional) association, taking the auxiliary variable to be the scalar $\Ubar$ and updating its distribution by conditioning on the observed values of the residuals, $U_i-\Ubar$, given in \eqref{eq:normal.mean.b}.  Of course, in this case, the mean and the residuals are independent, so this amounts to ignoring \eqref{eq:normal.mean.b}, and the (conditional) association can be rewritten as  
\[ \Xbar = \theta + \sigma n^{-1/2} \Phi^{-1}(U), \quad U \sim \unif(0,1). \]
This simplifies the IM construction, since the auxiliary variable is only a scalar.  

For the A-step, we have $\Theta_x(u) = \{\xbar-\sigma n^{-1/2} \Phi^{-1}(u)\}$, a singleton set.  If, for the P-step, we take the default predictive random set $\S$ in \eqref{eq:default.prs}, then the C-step combines $\Theta_x$ with $\S$ to get
\[ \Theta_x(\S) = \bigl[ \xbar - \sigma n^{-1/2}\Phi^{-1}(0.5 \pm |U-0.5|) \bigr], \quad U \sim \unif(0,1). \]
The belief and plausibility functions are now easy to evaluate.  For example, if $A=\{\theta\}$ is a singleton assertion, $\bel_x(A;\S)$ is trivially zero, but 
\[ \pl_x(\theta;\S) = 1 - \Bigl| 2\Phi\Bigl( \frac{\xbar-\theta}{\sigma n^{-1/2}} \Bigr) - 1 \Bigr|. \]
From this, it is easy to check that the corresponding $100(1-\alpha)$\% plausibility region, $\{\theta: \pl_x(\theta;\S) > \alpha\}$, matches up exactly with the classical z-interval.

\section{Marginal inferential models}
\label{S:mim} 

\subsection{Preview: normal mean problem, cont. }
\label{SS:preview}

Suppose $X_1,\ldots,X_n$ are independent $\nm(\mu,\sigma^2)$ observations, with $\theta=(\mu,\sigma^2)$ unknown, but only the mean $\mu$ is of interest.  Following the conditional IM argument in \citet{imcond}, the baseline (conditional) association for $\theta$ may be taken as 
\begin{equation}
\label{eq:normal.mean.baseline}
\Xbar = \mu + \sigma n^{-1/2} U_1 \quad \text{and} \quad S = \sigma U_2, 
\end{equation}
where $U_1 \sim {\sf N}(0,1)$ and $(n-1)U_2^2 \sim {\sf ChiSq}(n-1)$.  This is equivalent to defining an association for $(\mu,\sigma^2)$ based on the minimal sufficient statistic; see Section~4.1 of \citet{imcond}.  This association involves two auxiliary variables; but since there is effectively only one parameter, we hope, for the sake of efficiency, to reduce the dimension of the auxiliary variable.  We may equivalently write this association as 
\begin{equation}
\label{eq:normal.mean.baseline.alt}
\Xbar = \mu + S n^{-1/2} U_1/U_2 \quad \text{and} \quad S = \sigma U_2. 
\end{equation}
The second expression in the above display has the following property: for any $s$, $\mu$, and $u_2$, there exists a $\sigma$ such that $s = \sigma u_2$.  This implies that, since $\sigma$ is free to vary, there is no direct information that can be obtained about $\mu$ by knowing $U_2$.  Therefore, there is no benefit to retain the second expression in \eqref{eq:normal.mean.baseline.alt}---and eventually predict the corresponding auxiliary variable $U_2$---when $\mu$ is the only parameter of interest.  

An alternative way to look at this point is as follows.  When $(\xbar, s)$ is fixed at the observed value, since $\sigma$ can take any value, we know that the unobserved $(U_1,U_2)$ must lie on exactly one of the $u$-space curves
\[ \frac{u_1}{u_2} = \frac{n^{1/2}(\xbar - \mu)}{s}, \]
indexed by $\mu$.  In Section~\ref{SS:three.step}, we argued that the ``best possible inference'' on $(\mu,\sigma^2)$ obtains if we observed the pair $(U_1,U_2)$.  In this case, however, the ``best possible marginal inference'' on $\mu$ obtains if only we observed which of these curves $(U_1,U_2)$ lies on.  This curve is only a one-dimensional quantity, compared to the two-dimensional $(U_1,U_2)$, so an auxiliary variable dimension reduction is possible as a result of the fact that only $\mu$ is of interest.  In this particular case, we can ignore the $U_2$ component and work with the auxiliary variable $V=U_1/U_2$, whose marginal distribution is a Student-t with $n-1$ degrees of freedom.  As this is involves a one-dimensional auxiliary variable only, we have simplified the P-step without sacrificing efficiency; see Section~\ref{SS:optimal}. 

\subsection{Regular models and marginalization}
\label{SS:basics}

The goal of this section is to formalize the arguments given in Section~\ref{SS:preview} for the normal mean problem.  For $\theta=(\psi,\xi)$, with $\psi$ the interest parameter, the basic idea is to set up a new association between the data $X$, an auxiliary variable $W$, and the parameter of interest $\psi$ only.  With this, we can achieve an overall efficiency gain since the dimension of $W$ generally less than that of the original auxiliary variable.  

To emphasize that $\theta=(\psi,\xi)$, rewrite the association \eqref{eq:amodel} as 
\begin{equation}
\label{eq:new-aeqn}
p(X;\psi,\xi) = a(U;\psi,\xi), \quad U \sim \prob_U.
\end{equation}
Now suppose that there are functions $\bar p$, $\bar a$, and $c$, and new auxiliary variables $V=(V_1,V_2)$, with distribution $\prob_V$, such that \eqref{eq:new-aeqn} can equivalently be written as 
\begin{subequations}
\label{eq:regular1}
\begin{align}
\bar p(X,\psi) & = \bar a(V_1,\psi) \label{eq:regular1a} \\
c(X,V_2,\psi,\xi) & = 0. \label{eq:regular1b}
\end{align}
\end{subequations}
The equivalence we have in mind here is that a sample $X$ from the sampling model $\prob_{X|\psi,\xi}$, for given $(\psi,\xi)$, can be obtained by sampling $V=(V_1,V_2)$ from $\prob_V$ and solving for $X$.  See Remark~\ref{re:regular} below for more on the representation \eqref{eq:regular1}.  

The normal mean example in Section~\ref{SS:preview}, with association \eqref{eq:normal.mean.baseline} and auxiliary variables $(U_1,U_2)$, is clearly of the form \eqref{eq:regular1}, with $V_1=U_1/U_2$ and $V_2=U_2$.  In the normal example, recall that the key point leading to efficient marginal inference was that observing $U_2$ does not provide any direct information about the interest parameter $\mu$.  For the general case, we need to assume that this condition, which we call ``regularity,'' holds.  This assumption holds in many examples, but there are non-regular models and, in such cases, special considerations are required; see Section~\ref{S:gmim}. 

\begin{defn}
\label{def:regular}
The association \eqref{eq:new-aeqn} is regular if it can be written in the form \eqref{eq:regular1}, and the function $c$ satisfies, for any $(x,v_2,\psi)$, there exists a $\xi$ such that $c(x,v_2,\psi,\xi) = 0$.   
\end{defn}

In the regular case, it is clear that, like in the normal mean example in Section~\ref{SS:preview}, knowing the exact value of $V_2$ does not provide any information about the interest parameter $\psi$, so there is no benefit to retaining the component \eqref{eq:regular1b} and eventually trying to predict $V_2$.  Therefore, in the regular case, we propose to construct a marginal IM for $\psi$ with an association based only on \eqref{eq:regular1a}.  That is, 
\begin{equation}
\label{eq:marginal.assoc}
\bar p(X,\psi) = \bar a(V_1, \psi), \quad V_1 \sim \prob_{V_1}. 
\end{equation}
In regard to efficiency, like in the normal mean example of Section~\ref{SS:preview}, the key point here is that $V_1$ is generally of lower dimension than $U$, so, in the regular case, an auxiliary variable dimension reduction is achieved, thereby increasing efficiency.  

From this point, we can follow the three steps in Section~\ref{SS:three.step} to construct a marginal IM for $\psi$.  For the A-step, start with the marginal association \eqref{eq:marginal.assoc} and write
\[ \Psi_x(v_1) = \{\psi: \bar p(x,\psi) = \bar a(v_1,\psi) \}. \]
For the P-step, introduce a valid predictive random set $\S$ for $V_1$.  Combine these results in the C-step to get 
\[ \Psi_x(\S) = \bigcup_{v_1 \in \S} \Psi_x(v_1). \]
If $\Psi_x(\S) \neq \varnothing$ with $\prob_\S$-probability~1, then, for any assertion $A \subseteq \Psi$, the marginal belief and plausibility functions can be computed as follows:
\begin{align*}
\mbel_x(A;\S) & = \prob_{\S}\bigl\{ \Psi_x(\S) \subseteq A \bigr\}  \\
\mpl_x(A;\S) & = 1-\mbel_x(A^c;\S). 
\end{align*}
These functions can be used for inference as in Section~\ref{S:im.review}.  In particular, we may construct marginal plausibility intervals for $\psi$ using $\mpl_x$.  As we mentioned in Section~\ref{SS:three.step}, if $\Psi_x(\S) = \varnothing$ with positive $\prob_\S$-probability, then some adjustment to the belief function formula is needed, and this can be done by conditioning or by stretching.  The latter method, due to \citet{leafliu2012} is preferred; see Section~\ref{SS:stein2}.

\subsection{Remarks}
\label{SS:remarks}

\begin{remark}
\label{re:regular}
When there exists a one-to-one mapping $x \mapsto (T(x),H(x))$ such that the conditional distribution of $T(X)$, given $H(X)$, is free of $\xi$ and the marginal distribution of $H(X)$ is free of $\psi$, then an association of the form \eqref{eq:regular1} is available.  These considerations are similar to those presented in \citet{severini1999} and the references therein.  Specifically, suppose the distribution of the minimal sufficient statistic factors as $p(t \mid h, \psi) p(h \mid \xi)$, for statistics $T=T(X)$ and $H=H(X)$.  In this case, the observed value $h$ of $H$ provides no information about $\psi$, so we can take \eqref{eq:regular1a} to characterize the conditional distribution $p(t \mid h, \psi)$ of $T$, given $H=h$, and \eqref{eq:regular1b} to characterize the marginal distribution $p(h \mid \xi)$ of $H$.  Also, if $\prob_{X|\psi,\xi}$ is a composite transformation model \citep{bn1988}, then $T$ can be taken as a (maximal) $\xi$-invariant, whose distribution depends on $\psi$ only.  
\end{remark}

\begin{remark}
\label{re:bayes}
Consider a Bayes model with a genuine (not default) prior distribution $\Pi$ for $(\psi, \xi)$.  In this case, we can write an association, in terms of an auxiliary variable $U=(U_0,U_\psi,U_\xi) \sim \prob_{U_0} \times \Pi$, as 
\[ (\psi,\xi) = (U_\psi, U_\xi), \quad p(X; \psi, \xi) = a(U_0; U_\psi, U_\xi). \]
According to the argument in \citet[][Remark~4]{imcond}, an initial dimension reduction obtains, so that the baseline association can be re-expressed as 
\[ (\psi, \xi) = (U_\psi, U_\xi), \quad (U_\psi, U_\xi) \sim \Pi_X, \]
where $\Pi_X$ is just the usual posterior distribution of $(\psi,\xi)$, given $X$, obtained from Bayes theorem.  Now, by splitting the posterior into the appropriate marginal and conditional distributions, we get a decomposition 
\[ \psi = U_\psi \quad \text{and} \quad \xi - U_\xi = 0, \quad (U_\psi,U_\xi) \sim \Pi_X. \]
This association is regular, so a marginal association for $\psi$ obtains from its marginal posterior distribution, completing the A-step.  The P-step is a bit different in this case and deserves some explanation.  When no meaningful prior distribution is available, there is no probability space on which to carry out probability calculations.  The use of a random set on the auxiliary variable space provides such a space, and validity ensures that the corresponding belief and plausibility calculations are meaningful.  However, if a genuine prior distribution is available, then there is no need to introduce an auxiliary variable space, random sets, etc.  Therefore, in the genuine-prior Bayes context, we can take a singleton predictive random set in the P-step.  Then the IM belief function obtained in the C-step is exactly the usual Bayesian posterior distribution.  
\end{remark}

\begin{remark}
\label{re:other.marginal}
The discussion in Section~\ref{SS:basics}, and the properties to be discussed in the coming sections, suggest that the baseline association, or sampling model, being regular in the sense of Definition~\ref{def:regular} is a sufficient condition for valid marginalization.  In such cases, like in Sections~\ref{SS:bivariate}--\ref{SS:creasy}, fiducial and objective Bayes methods are valid for certain assertions, and the marginal IM will often give the same answers; examples with (implicit or explicit) parameter constraints, like that in Section~\ref{SS:stein2}, reveal some differences between IMs and fiducial and Bayes.  However, in non-regular problems, Bayes and fiducial marginalization may not be valid; see Section~\ref{S:gmim}.  In this light, regularity appears to also be a necessary condition for valid marginal inference.
\end{remark}

\begin{remark}
\label{re:take.away}
Condition \eqref{eq:regular1} helps to formalize the discussion in \citet[][Example~5.1]{hannig2012} for marginalization in the fiducial context by characterizing the set of problems for which his manipulations can be carried out.  Perhaps more importantly, in light of the observations in Remark~\ref{re:other.marginal}, it helps to explain why, for valid prior-free probabilistic marginal inference, the marginalization step must be considered \emph{before} producing evidence measures on the parameter space.  Indeed, producing fiducial or objective Bayes posterior distributions for $(\psi,\xi)$ and then marginalizing to $\psi$ is not guaranteed to work.  The choice of data-generating equation or reference prior must be considered before actually carrying out the marginalization.  The same is true for IMs, though identifying the relevant directions in the auxiliary variable space, as discussed in the previous sections, is arguably more natural than, say, constructing reference priors.  
\end{remark}

\subsection{Validity of regular marginal IMs}
\label{SS:valid}

An important question is if, for suitable $\S$, the marginal IM is valid in the sense of Definition~\ref{def:validity}.  We give the affirmative answer in the following theorem.  

\begin{thm}
\label{thm:credible}
Suppose that the baseline association \eqref{eq:new-aeqn} is regular in the sense of Definition~\ref{def:regular}, and that $\S$ is a valid predictive random set for $V_1$ with the property that $\Psi_x(\S) \neq \varnothing$ with $\prob_{\S}$-probability~1 for all $x$.  Then the marginal IM is valid in the sense of Definition~\ref{def:validity}, that is, for any $A \subset \Psi$ and any $\alpha \in (0,1)$, the marginal belief function satisfies
\[ \sup_{(\psi,\xi) \in A^c \times \Xi} \prob_{X|(\psi,\xi)}\bigl\{ \mbel_X(A;\S) \geq 1-\alpha \bigr\} \leq \alpha. \]
Since this holds for all $A$, a version of \eqref{eq:valid.pl} also holds:
\[ \sup_{(\psi,\xi) \in A \times \Xi} \prob_{X|(\psi,\xi)}\bigl\{ \mpl_X(A;\S) \leq \alpha \bigr\} \leq \alpha. \]
\end{thm}

\begin{proof}
\ifthenelse{1=1}{
Similar to the validity theorem proofs in \citet{imbasics, imbasics.c, imcond}.  This result is also covered by the proof of Theorem~\ref{thm:unif.valid.im} below. 
}{ 
Define $Q_{\S}(w) = \prob_{\S}\{\S \not\ni w\}$, the probability that the predictive random set $\S$ misses the target $w$.  Theorems~1--2 in \citet{imbasics} show that admissibility of $\S$ implies that $\prob_W\{Q_{\S}(W) \geq 1-\alpha \} \leq \alpha$ for all $\alpha \in (0,1)$.  

Note that, if $\psi$ is the true parameter value, then $\Psi_x(\S)$ misses $\psi$ if and only if $\S$ misses $w^\star$.  Suppose $\psi \not\in A$.  Then monotonicity of the marginal belief function gives  
\[ \mbel_x(A;\S) \leq \mbel_x(\{\psi\}^c;\S) = \prob_{\S}\{\Psi_x(\S) \not\ni \psi\} = \prob_{\S}\{\S \not\ni w^\star\} = Q_{\S}(w^\star). \]
Note that $x$ and $w^\star$ are connected in the sense that a sample $X \sim \prob_{X|\theta}$ determines and is determined by a sample $W \sim \prob_W$.  So varying $X$ according to $\prob_{X|\theta}$ is equivalent to varying $W=W^\star$ according to $\prob_W$.  Therefore, we get the desired result: for all $\psi \not\in A$, 
\[ \prob_{X|(\psi,\xi)}\{ \mbel_X(A;\S) \geq 1-\alpha \} \leq \prob_W\{Q_{\S}(W) \geq 1-\alpha\} \leq \alpha, \]
where the last inequality follows by the admissibility assumption.  The second part of the theorem follows from the fact that $\mpl_x(A;\S) = 1-\mbel_x(A^c;\S)$.  
}
\end{proof}

Therefore, if the baseline association is regular and the predictive random set is valid, then the marginal IM constructed has the desirable frequency calibration property.  In particular, this means that marginal plausibility intervals based on $\mpl_x$ will achieve the nominal frequentist coverage probability; we see this exactness property in the examples in Section~\ref{S:applications}.  More importantly, this validity property ensures that the output of the marginal IM is meaningful both within and across experiments.

\subsection{Efficiency of regular marginal IMs}
\label{SS:optimal}

Start with a regular association \eqref{eq:regular1}, and let $\S \sim \prob_\S$ be a valid predictive random set for $(V_1,V_2)$ in $\VV_1 \times \VV_2$.  Assume that $\Theta_x(\S)$, a random subset of $\Psi \times \Xi$, is non-empty with $\prob_\S$-probability~1 for all $x$.  This assumption holds, for example, if the dimension of $(V_1,V_2)$ matches that of $(\psi,\xi)$, which can often be achieved by applying the techniques in \citet{imcond} to the baseline association prior to marginalization.  

In the regular case, we have proposed to marginalize over $\xi$ by ``ignoring'' the second component of the association \eqref{eq:regular1b}.  We claim that an alternative way to view the marginalization step is via an appropriate stretching of the predictive random set.  To see this, for given $\S$, consider an enlarged predictive random set $\bar\S$ obtained by stretching $\S$ to fill up the entire $V_2$-dimension.  Equivalently, take $\S_1$ to be the projection of $\S$ to the $V_1$-dimension, and then set $\bar\S = \S_1 \times \VV_2$.  It is clear that, if $\S$ is valid, then so is $\bar\S$, but, as explained in the next paragraph, the larger $\bar\S$ cannot be more efficient than $\S$.  In the case of marginal inference, however, the bigger predictive random set $\bar\S$ is equivalent to $\S$, so stretching/ignoring yields valid marginalization without loss of efficiency; see, also, the discussion below following Theorem~\ref{thm:efficient}. 

We pause here to give a quick high-level review of the notion of IM efficiency as presented in \citet{imbasics}.  In the present context, we have two valid predictive random sets $\S$ and $\bar\S$.  If $B$ is some assertion about $(\psi,\xi)$, then we say that $\S$ is as efficient as $\bar\S$ (relative to $B$) if $\pl_X(B;\S)$ is stochastically no larger than $\pl_X(B;\bar\S)$ as a function of $X$ when $(\psi,\xi) \not\in B$.  To see the intuition behind this definition, note that a plausibility region for $(\psi,\xi)$ is obtained by keeping those singleton assertions whose plausibility exceeds a threshold.  Making the plausibility function as small as possible, subject to the validity constraint, will make the plausibility region smaller and, hence, the inference more efficient.  One quick application, for the present case where $\S \subset \bar\S$, note that $\Theta_x(\S) \subseteq \Theta_x(\bar\S)$.  Since the bigger set will have a higher probability of intersecting with an assertion, $\pl_x(\cdot;\bar\S) \geq \pl_x(\cdot;\S)$ and $\bar\S$ is no more efficient than $\S$.  

For marginal inference, the assertions of interest are of the form $B=A \times \Xi$, where $A \subseteq \Psi$ is the marginal assertion on $\psi$.  In this case, it is easy to check that 
\[ \pl_x(A \times \Xi; \bar\S) = \pl_x(A \times \Xi; \S), \quad \text{for all $x$}, \]
so $\S$ and $\bar\S$ are equivalent for inference on $A \times \Xi$.  That is, no efficiency is lost by stretching the predictive random set $\S$ to $\bar\S$.  A natural question is: why stretch $\S$ to $\bar\S$?  To answer this, note that $\pl_x(A \times \Xi; \bar\S)$ is exactly equal to $\mpl_x(A;\S_1)$, the marginal plausibility function for $\psi$, at $A$, based on the projection $\S_1$ of the random set $\S$ to the $V_1$-dimension.  Therefore, we see that ``ignoring'' the auxiliary variable $V_2$ in \eqref{eq:regular1b} is equivalent to stretching the predictive random set for $(V_1,V_2)$ till it fills the $V_2$-dimension.  We summarize this discussion in the following theorem.

\begin{thm}
\label{thm:efficient}
Consider a regular association \eqref{eq:regular1} and let $\S$ be a valid predictive random set for $(V_1,V_2)$ such that $\Theta_x(\S) \neq \varnothing$ with $\prob_\S$-probability~1 for each $x$.  Then the projection $\S_1$ is a valid predictive random set for $V_1$ and, furthermore, $\pl_x(A \times \Xi; \S) = \mpl_x(A; \S_1)$ for any marginal assertion $A \subseteq \Psi$.  
\end{thm}

The theorem says that marginalization by ignoring \eqref{eq:regular1b} or stretching a joint predictive random set to fill the $V_2$-dimension results in no loss of efficiency.  The point is that any valid predictive random set for $(V_1,V_2)$ will result in a valid marginal predictive random set for $V_1$.  However, there are advantages to skipping directly to specification of the marginal predictive random set for $V_1$.  First, the lower-dimensional auxiliary variable is easier to work with.  Second, a good predictive random set for $V_1$ will tend to be smaller than the corresponding projection down from a good predictive random set for $(V_1,V_2)$.  Figure~\ref{fig:efficiency} gives an example of this for $(V_1,V_2)$ iid $\nm(0,1)$.  Realizations of two predictive random sets (both with coverage probability 0.5) are shown.  Clearly, constructing an interval directly for $V_1$ and then stretching over $V_2$ (rectangle) is more efficient than the projection of the joint predictive random set for $(V_1,V_2)$.  So, by Theorem~\ref{thm:efficient}, a valid joint IM will lead to a valid marginal IM, but marginalization before constructing a predictive random set, as advocated in Section~\ref{SS:basics}, is generally more efficient.  


\begin{figure}
\begin{center}
\scalebox{0.6}{\includegraphics{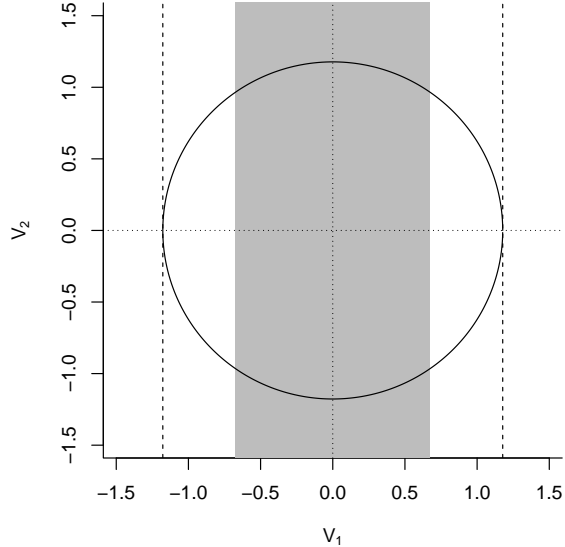}}
\end{center}
\caption{Realizations of two predictive random sets for $(V_1,V_2)$ iid $\nm(0,1)$; both have coverage probability 0.5.  Circle is for the pair, and the projection down to $V_1$-dimension is shown with dashed lines; gray rectangle is the cylinder obtained by constructing a predictive random set directly for $V_1$.}
\label{fig:efficiency}
\end{figure}

\section{Examples}
\label{S:applications}

\subsection{Bivariate normal correlation}
\label{SS:bivariate}

Let $X_1,\ldots,X_n$, with $X_i = (X_{i1},X_{i2})$, $i=1,\ldots,n$, be an independent sample from a bivariate normal distribution with marginal means and variances $\xi=(\mu_1,\mu_2,\sigma_1^2,\sigma_2^2)$ and correlation $\psi$.  It is well known that $(\hat\mu_1,\hat\mu_2,\hat\sigma_1^2,\hat\sigma_2^2,\hat\psi)$, with $\hat\mu_j$ and $\hat\sigma_j^2$, $j=1,2$, the marginal sample means and variances, respectively, and 
\[ \hat\psi = \frac{\sum_{i=1}^n (X_{i1}-\Xbar_1)(X_{i2}-\Xbar_2)}{\bigl\{\sum_{i=1}^n (X_{i1}-\Xbar_1)^2 \sum_{i=1}^n (X_{i2}-\Xbar_2)^2\bigr\}^{1/2}}, \]
the sample correlation, together form a joint minimal sufficient statistic.  The argument in \citet[][Sec.~4.1]{imcond} implies that the conditional IM for $\theta=(\psi,\xi)$ can be expressed in terms of this minimal sufficient statistic.  That is, following the initial conditioning step, our baseline association for $(\psi,\xi)$ looks like 
\[ \hat\mu_j = \mu_j + \sigma_j M_j, \quad \hat\sigma_j^2 = \sigma_j^2 V_j, \quad \hat\psi=a(C,\psi), \quad j=1,2, \]
for an appropriate collection of auxiliary variables $U=(M_1,M_2,V_1,V_2,C)$ and a function $a(\cdot,\psi)$ to be specified below in \eqref{eq:bvn.fiducial}.  This is clearly a regular association, so we get a marginal association for $\psi$, which is most easily expressed as 
\begin{equation}
\label{eq:bvn.fiducial}
\hat\psi = G_\psi^{-1}(W), \quad W \sim \unif(0,1), 
\end{equation}
where $G_\psi$ is the distribution function of the sample correlation.  Fisher developed fiducial intervals for $\psi$ based on the fiducial density $p(\psi \mid \hat\psi) = |\del G_\psi(\hat\psi)/\del\psi|$.  In particular, the middle $1-\alpha$ region of this distribution is a $100(1-\alpha)$\% interval estimate for $\psi$.  It is known that this fiducial interval is exact, and also corresponds to the marginal posterior for $\psi$ under the standard objective Bayes prior for $\theta=(\mu_1,\mu_2,\sigma_1^2,\sigma_2^2,\psi)$; see \citet{berger2006}.  Interestingly, there is no proper Bayesian prior with the fiducial distribution $p(\psi \mid \hat\psi)$ as the posterior.  However, it is easy to check that, with the default predictive random set \eqref{eq:default.prs} for $U$ in \eqref{eq:bvn.fiducial}, the corresponding marginal plausibility interval for $\psi$ corresponds exactly to the classical fiducial interval.

\subsection{Ratio of normal means}
\label{SS:creasy}

Let $X_1 \sim \nm(\psi\xi,1)$ and $X_2 \sim \nm(\xi,1)$ be two independent normal samples, with unknown $\theta=(\psi,\xi)$, and suppose the goal is inference on $\psi$, the ratio of means.  This is the simplest version of the Fieller--Creasy problem \citep{fieller1954, creasy1954}.  Problems involving ratios of parameters, such as a gamma mean (Section~\ref{SS:gamma.mean}), are generally quite challenging and require special considerations.     

To start, write the baseline association as 
\[ X_1 = \psi\xi + U_1 \quad \text{and} \quad X_2 = \xi + U_2, \quad U_1,U_2 \iid \nm(0,1). \]
After a bit of algebra, this is clearly equivalent to 
\[ \frac{X_1 - \psi X_2}{(1 + \psi^2)^{1/2}} = \frac{U_1-\psi U_2}{(1+\psi^2)^{1/2}} \quad \text{and} \quad X_2 - \xi - U_2 = 0. \]
We, therefore, have a regular decomposition \eqref{eq:regular1}, so ``ignoring'' the part involving $\xi$ gives the marginal association
\[ \frac{X_1 - \psi X_2}{(1 + \psi^2)^{1/2}} = V, \]
where $V = (U_1-\psi U_2)/(1+\psi^2)^{1/2}$.  Since $V$ is a pivot, i.e., $V \sim \nm(0,1)$ for all $\psi$, the marginal association can be expressed as
\[ \frac{X_1 - \psi X_2}{(1 + \psi^2)^{1/2}} = \Phi^{-1}(W), \quad W \sim \unif(0,1). \]
For the corresponding marginal IM, the A-step gives 
\[ \Psi_x(w) = \{\psi: x_1 - \psi x_2 = (1 + \psi^2)^{1/2} \Phi^{-1}(w)\}. \]
Note that this problem has a non-trivial constraint, i.e., $\Psi_x(w)$ is empty for some $(x,w)$ pairs.  A similar issue arises in the many-normal-means problem in Section~\ref{SS:stein2}.  
However, if $\S$ is symmetric about $w=0.5$, like the default \eqref{eq:default.prs}, then $\Psi_x(\S)$ is non-empty with $\prob_\S$-probability 1 for all $x$.  Therefore, the marginal IM is valid if $\S$ is valid and symmetric about $w=0.5$.  In fact, for $\S$ in \eqref{eq:default.prs}, the marginal plausibility intervals for $\psi$ are the same as the confidence interval proposed by Fieller.  Then validity of our marginal IM provides an alternative proof of the coverage properties of Fieller's interval.

\subsection{Many-normal-means}
\label{SS:stein2}

Suppose $X \sim \nm_n(\theta,I_n)$, where $X$ and $\theta$ are the $n$-vectors of observations and means, respectively.  Assume $\theta \neq 0$ and write $\theta$ as $(\psi,\xi)$, where $\psi = \|\theta\|$ is the length of $\theta$ and $\xi = \theta/\|\theta\|$ is the unit vector in the direction of $\theta$.  The goal is to make inference on $\psi$.    

A baseline association for this problem is 
\begin{equation}
\label{eq:mnm.baseline}
X = \psi \xi + U, \quad U \sim \nm_n(0,I_n). 
\end{equation}
Since $\psi$ is a scalar, it is inefficient to predict the $n$-vector $U$.  Fortunately, the marginalization strategy discussed above will yield a lower dimensional auxiliary variable.     

Take $M$ to be an orthonormal matrix with $\xi^\top$ as its first row, and define $V = M U$.    This transformation does not alter the distribution, i.e., both $U$ and $V$ are ${\sf N}_n(0,I_n)$, and the baseline association \eqref{eq:mnm.baseline} can be re-expressed as  
\[ \|X\|^2 = (\psi + V_1)^2 + \|V_{2:n}\|^2 \quad \text{and} \quad \frac{X - M^{-1}V}{\|X-M^{-1}V\|} = \xi. \]
This is of the regular form \eqref{eq:regular1}, so the left-most equation above gives a marginal IM for $\psi$.  We make one more change of auxiliary variable, $W = F_{n,\psi}\bigl( (\psi + V_1)^2 + \|V_{2:n}\|^2 \bigr)$, where $F_{n,\psi}$ is the distribution function of a non-central chi-square with $n$ degrees of freedom and non-centrality parameter $\psi^2$.  The new marginal association is $\|X\|^2 = F_{n,\psi}^{-1}(W)$, with $W \sim \unif(0,1)$, and the A-, P-, and C-steps can proceed as usual.  In particular, for the P-step, we can use the predictive random set $\S$ in \eqref{eq:default.prs}.  However, the set $\Psi_x(w) = \{\psi: \|x\|^2 = F_{n,\psi}(w)\}$ is empty for $w$ in a set of positive measure so, if we take $\S$ to as in \eqref{eq:default.prs}, then the conditions of Theorem~\ref{thm:credible} are violated.  To remedy this, and construct a valid marginal IM, the preferred technique is to use an elastic predictive random set \citep{leafliu2012}.  These details, discussed next, are interesting and allow us to highlight a particular difference between IMs and fiducial.  

The default predictive random set $\S$ is determined by a sample $W \sim \unif(0,1)$, so write $\S=\S_W$.  Then $\Psi_x(\S_W)$ is empty if and only if 
\[ F_{n,0}(\|x\|^2) < \tfrac12 \quad \text{and} \quad |W-\tfrac12| < \tfrac12 - F_{n,0}(\|x\|^2). \]
To avoid this, \citet{leafliu2012} propose to stretch $\S_W$ just enough that $\Psi_x(\S_W)$ is non-empty.  In the case where $F_{n,0}(\|x\|^2) < \frac12$, we have 
\[ \Psi_x(\S_W) = \begin{cases} \{0\} & \text{if $|W-\frac12| \leq \frac12 - F_{n,0}(\|x\|^2)$}, \\ \{\psi: |F_{n,\psi}(\|x\|^2) - \frac12| \leq |W - \frac12|\} & \text{otherwise}. \end{cases} 
\]
So, for the case $F_{n,0}(\|x\|^2) < \frac12$, the marginal plausibility function based on the elastic version of the default predictive random set is $\mpl_x(0)=1$ and, for $\psi > 0$, 
\[ \mpl_x(\psi) = 2F_{n,0}(\|x\|^2) - \max\{|2F_{n,\psi}(\|x\|^2) - 1| + 2F_{n,0}(\|x\|^2) - 1,0\}; \]
the corresponding $100(1-\alpha)$\% marginal plausibility interval is $[0,\hat\psi(x,\frac{\alpha}{2})]$, where $\hat\psi(x,q)$ solves the equation $F_{n,\psi}(\|x\|^2) = q$.  When $F_{n,0}(\|x\|^2) > \frac12$, $\Psi_x(\S_W)$ is non-empty with probability~1, so no adjustments to the default predictive random set are needed.  In that case, the marginal plausibility function is 
\[ \mpl_x(\psi) = 1 - |2F_{n,\psi}(\|x\|^2) - 1|, \]
and the corresponding $100(1-\alpha)$\% plausibility interval for $\psi$ is 
\[ \{\psi: \tfrac{\alpha}{2} \leq F_{n,\psi}(\|x\|^2) \leq 1 - \tfrac{\alpha}{2}\}. \]
\citet{leafliu2012} show that an IM obtained using elastic predictive random sets are valid, so the coverage probability of the plausibility interval is $1-\alpha$ for all $\psi$.  

For comparison, the generalized fiducial approach outlined in Example~5.1 of \citet{hannig2012} uses the a same association $\|X\|^2 = F_{n,\psi}^{-1}(W)$ with $W \sim \unif(0,1)$.  Like in the discussion above, something must be done to avoid the conflict cases.  Following the standard fiducial ``continue to believe'' logic, it seems that one should condition on the event $\{W \leq F_{n,0}(\|x\|^2)\}$.  In this case, the corresponding central $100(1-\alpha)$\% fiducial interval for $\psi$ is 
\[ \Bigl\{ \psi: \frac{\alpha}{2} \leq \frac{F_{n,\psi}(\|x\|^2)}{F_{n,0}(\|x\|^2)} \leq 1 - \frac{\alpha}{2} \Bigr\}. \]
The recommended objective Bayes approach, using the prior $\pi(\theta) \propto \|\theta\|^{-(n-1)}$ \citep[e.g.,][]{tibshirani1989}, is more difficult computationally but its credible intervals are similar to those in the previous display.  However, this fiducial approach is less efficient compared to the marginal IM solution presented above.  Alternative fiducial solutions, such as that in \citet[][Example~5]{hannig.iyer.patterson.2006}, which employ a sort of stretching, similar to our use of an elastic predictive random set, give results comparable to our marginal IM.

\ifthenelse{1=1}{}{

We claim that the marginal IM describe above is no less efficient than the fiducial method.  To justify this, we present a brief simulation study.  For nine different $(n,\psi)$ pairs, 10,000 data sets are simulated, and 95\% interval estimates for $\psi$ are constructed using each method.  The interval estimates are compared based on estimated coverage probability and mean length in Table~\ref{table:mnm}.  Notice that all of the IM coverage probabilities are within an acceptable range of the target 0.95, while the fiducial only reaches the nominal level when $\psi$ is relatively large.  Moreover, the marginal IM intervals are shorter than the fiducial intervals, on average, across all the simulation settings.  For values $\psi > 1$, the two methods give similar answers.  The recommended objective Bayes approach, using the reference prior $\pi(\theta) \propto \|\theta\|^{-(n-1)}$ \citep[e.g.,][]{tibshirani1989}, is slightly more difficult computationally and gives results similar to fiducial in this case.

\begin{table}
\caption{Estimated coverage probabilities and mean lengths for the 95\% marginal IM plausibility and fiducial confidence intervals for $\psi$.}
\label{table:mnm}
\begin{center}
\begin{tabular}{cccccc}
\hline
& & \multicolumn{2}{c}{Coverage probability} & \multicolumn{2}{c}{Mean length} \\
$n$ & $\psi$ & Marginal IM & Fiducial & Marginal IM & Fiducial \\
\hline
2 & 0.1 & 0.947 & 0.000 & 2.80 & 2.98 \\
& 0.5 & 0.950 & 0.910 & 2.88 & 3.01 \\
& 1.0 & 0.951 & 0.970 & 3.09 & 3.12 \\
5 & 0.1 & 0.949 & 0.000 & 3.15 & 3.31 \\
& 0.5 & 0.949 & 0.831 & 3.24 & 3.34 \\
& 1.0 & 0.950 & 0.958 & 3.43 & 3.44 \\
10 & 0.1 & 0.950 & 0.000 & 3.53 & 3.65 \\
& 0.5 & 0.946 & 0.738 & 3.57 & 3.66 \\
& 1.0 & 0.951 & 0.943 & 3.73 & 3.73 \\
\hline
\end{tabular}
\end{center}
\end{table}
}

\section{Marginal IMs for non-regular models}
\label{S:gmim}

\subsection{Motivation and a general strategy}
\label{SS:strategy}

The previous sections focused on the case where the sampling model admits a regular baseline association \eqref{eq:regular1}.  However, there are important problems which are not regular, e.g., Sections~\ref{SS:behrens.fisher} and \ref{SS:gamma.mean} below, and new techniques are needed for such cases.  

Our strategy here is based on the idea of marginalization via predictive random sets.  That is, marginalization can be accomplished by using a predictive random set for $U$ in the baseline association that spans the entire ``non-interesting'' dimension in the transformed auxiliary variable space.  But since we are now focusing on the non-regular model case, some further adjustments are needed.  Start with an association of the form
\begin{subequations}
\label{eq:regular2}
\begin{align}
\bar p(X;\psi) & = Z_1(\xi), \label{eq:regular2a} \\
c(X,Z_2,\psi,\xi) & = 0. \label{eq:regular2b}
\end{align}
\end{subequations}
See the examples in Section~\ref{SS:nonregular.examples}.  This is similar to \eqref{eq:regular1} except that the distribution of $Z_1(\xi)$ above depends on the nuisance parameter $\xi$.  Therefore, we cannot hope to eliminate $\xi$ by simply ``ignoring'' \eqref{eq:regular2b} as we did in the regular case.  

If $\xi$ were known, then a valid predictive random set could be introduced for $Z_1(\xi)$.  But this predictive random set for $Z_1(\xi)$ would generally be valid only for the particular $\xi$ in question.  Since $\xi$ is unknown in our context, this predictive random set would need to be enlarged in order to retain validity for all possible $\xi$ values.  This suggests a notion of \emph{uniformly valid} predictive random sets.

\begin{defn}
\label{def:unif.valid.prs}
A predictive random set $\S$ for $Z_1(\xi)$ is \emph{uniformly valid} if it is valid for all $\xi$, i.e., $\prob_\S\{\S \not\ni Z_1(\xi)\}$ is stochastically no larger than $\unif(0,1)$, as a function of $Z_1(\xi) \sim \prob_{Z_1(\xi)}$, for all $\xi$.  
\end{defn}

For this more general non-regular model case, we have an analogue of the validity result in Theorem~\ref{thm:credible} for the case of a regular association.  We shall refer to the resulting IM for $\psi$ as a \emph{generalized marginal IM}.  

\begin{thm}
\label{thm:unif.valid.im}
Let $\S$ be a uniformly valid predictive random set for $Z_1(\xi)$ in \eqref{eq:regular2a}, such that $\Psi_x(\S) \neq \varnothing$ with $\prob_\S$-probability~1 for all $x$.  Then the corresponding generalized marginal IM for $\psi$ is valid in the sense of Theorem~\ref{thm:credible} for all $\xi$.  
\end{thm}

\begin{proof}
The assumed representation of the sampling model $X \sim \prob_{X|(\psi,\xi)}$ implies that there exists a corresponding $Z_1(\psi) \sim \prob_{Z_1(\xi)}$.  That is, probability calculations with respect to the distribution of $X$ and are equivalent to probability calculations with respect to the distribution of $Z_1(\xi)$.  Take $\psi \not\in A$, so that $A \subseteq \{\psi\}^c$.  Then we have 
\[ \mbel_X(A;\S) \leq \mbel_X(\{\psi\}^c; \S) = \prob_\S\{\Psi_X(\S) \not\ni \psi\} = \prob_\S\{\S \not\ni Z_1(\xi)\}. \]
The assumed uniform validity implies that the right-hand side is stochastically no larger than $\unif(0,1)$ as a function of $Z_1(\xi)$.  This implies the same of the left-hand side as a function of $X$.  Therefore, 
\[ \sup_{\psi \not\in A} \prob_{X|(\psi,\xi)}\{\mbel_X(A;\S) \geq 1-\alpha\} \leq \alpha, \quad \forall \; \alpha \in (0,1). \]
This holds for all $\xi$, completing the proof.  
\end{proof}

There are a variety of ways to construct uniformly valid predictive random sets, but here we present just one idea, which is appropriate for singleton assertions and construction of marginal plausibility intervals.  Efficient inference for other kinds of assertions may require different considerations.  Our strategy here is based on the idea of replacing a nuisance parameter-dependent auxiliary variable $Z_1(\xi)$ by a nuisance parameter-independent type of stochastic bound.   

\begin{defn}
\label{def:fatter}
Let $Z$ be a random variable with distribution function $F_Z$ and median zero.  Another random variable $Z^\star$, with distribution function $F_{Z^\star}$ and median zero is said to be \emph{stochastically fatter} than $Z$ if the distribution function satisfy the constraint
\[ F_Z(z) < F_{Z^\star}(z), \quad z < 0 \quad \text{and} \quad F_Z(z) > F_{Z^\star}(z), \quad z > 0. \]
That is, $Z^\star$ has heavier tails than $Z$ in both directions.  
\end{defn}

For example, the Student-t distribution with $\nu$ degrees of freedom is stochastically fatter than ${\sf N}(0,1)$ for all finite $\nu$.  Two more examples are given in Section~\ref{SS:nonregular.examples}.  

Our proposal for constructing a generalized marginal IM for $\psi$ is based on the following idea: get a uniformly valid predictive random set for $Z_1(\xi)$ by first finding a new auxiliary variable $Z_1^\star$ that is stochastically fatter than $Z_1(\xi)$ for all $\xi$, and then introducing an ordinarily valid predictive random set for $Z_1^\star$.  The resulting predictive random set for $Z_1(\xi)$ is necessarily uniformly valid.  In this way, marginalization in non-regular models can be achieved through the choice predictive random set.  

Since the dimension-reduction techniques presented in the regular model case may be easier to understand and implement, it would be insightful to formulate the non-regular problem in this way also.  For this, consider replacing \eqref{eq:regular2a} with 
\begin{equation}
\label{eq:regular3a}
p(X,\psi) = Z_1^\star. 
\end{equation}
After making this substitution, the two parameters $\psi$ and $\xi$ have been separated in \eqref{eq:regular3a} and \eqref{eq:regular2b}, so the decomposition is regular.  Then, by the theory above, the marginal IM should now depend only on \eqref{eq:regular3a}.  The key idea driving this strategy is that a valid predictive random set for $Z_1^\star$ is necessarily uniformly valid for $Z_1(\xi)$.

\subsection{Examples}
\label{SS:nonregular.examples}

\subsubsection{Behrens--Fisher problem}
\label{SS:behrens.fisher}

The Behrens--Fisher problem is a fundamental one \citep{scheffe1970}.  It concerns inference on the difference between two normal means, based on two independent samples, when the standard deviations are completely unknown.  It turns out that there are no exact tests/confidence intervals that do not depend on the order in which the data is processed.  Standard solutions are given by \citet{hsu1938} and \citet{scheffe1970}, and various approximations are available, e.g., \citet{welch1938,welch1947}.  For a review of these and other procedures, see \citet{kimcohen1998}, \citet{ghosh.kim.2001}, and \citet{fraser.wong.sun.2009}.  

Suppose independent samples $X_{11},\ldots,X_{1n_1}$ and $X_{21},\ldots,X_{2n_2}$ are available from the populations ${\sf N}(\mu_1,\sigma_1^2)$ and ${\sf N}(\mu_2,\sigma_2^2)$, respectively.  Summarize the data sets with $\Xbar_k$ and $S_k$, $k=1,2$, the respective sample means and standard deviations.  The parameter of interest is $\psi = \mu_2-\mu_1$.  The problem is simple when $\sigma_1$ and $\sigma_2$ are known, or unknown but proportional.  For the general case, however, there is no simple solution.  Here we derive a generalized marginal IM solution for this problem.   

The basic sampling model is of the location-scale variety, so the general results in \citet{imcond} suggest that we may immediately reduce to a lower-dimensional model based on the sufficient statistics.  That is, we take as our baseline association 
\begin{equation}
\label{eq:bf-aeqn1}
\Xbar_k = \mu_k + \sigma_k \, n_k^{-1/2} \, U_{1k}, \quad \text{and} \quad S_k = \sigma_k U_{2k}, \quad k=1,2, 
\end{equation}
where the auxiliary variables are independent with $U_{1k} \sim {\sf N}(0,1)$ and $(n_k-1)U_{2k}^2 \sim {\sf ChiSq}(n_k-1)$ for $k=1,2$.  To incorporate $\psi = \mu_2-\mu_1$, combine the set of constraints in \eqref{eq:bf-aeqn1} for $\mu_1$ and $\mu_2$ to get $\Ybar = \psi + \sigma_2 \, n_2^{-1/2} \, U_{12} - \sigma_1 \, n_1^{-1/2} \, U_{11}$, where $\Ybar = \Xbar_2-\Xbar_1$.  Define $f(\sigma_1,\sigma_2) = (\sigma_1^2/n_1 + \sigma_2^2/n_2)^{1/2}$ and note that
\[ \sigma_2 \, n_2^{-1/2} \, U_{12} - \sigma_1 \, n_1^{-1/2} \, U_{11} \quad \text{and} \quad f(\sigma_1,\sigma_2) \, U_1 \]
are equal in distribution, where $U_1 \sim {\sf N}(0,1)$.  Making a change of auxiliary variables leads to a new and simpler baseline association for the Behrens--Fisher problem:
\begin{equation}
\label{eq:bf-aeqn2}
\Ybar = \psi + f(\sigma_1,\sigma_2) \, U_1 \quad \text{and} \quad S_k = \sigma_k U_{2k}, \quad k=1,2.
\end{equation}
Since $S_k = \sigma_k U_{2k}$ for $k=1,2$, we may next rewrite \eqref{eq:bf-aeqn2} as 
\begin{equation}
\label{eq:bf-aeqn3}
\frac{\Ybar - \psi}{f(S_1,S_2)} = \frac{f(\sigma_1, \sigma_2)}{f(\sigma_1 U_{21}, \sigma_2 U_{22})} \, U_1, \quad \text{and} \quad S_k = \sigma_k U_{2k}, \quad k=1,2. 
\end{equation}
If the right-hand side were free of $(\sigma_1,\sigma_2)$, the association would be regular and we could apply the techniques presented in Sections~\ref{S:mim} and \ref{S:applications}.  Instead, we have a decomposition like \eqref{eq:regular2}, so we shall follow the ideas presented in Section~\ref{SS:strategy}.  

Towards this, define $\xi = \xi(\sigma_1,\sigma_2) = (1 + n_1\sigma_1^2 / n_2 \sigma_2^2)^{-1}$, which takes values in $(0,1)$.  Make another change of auxiliary variable 
\[ Z_1(\xi) = \frac{U_1}{\{\xi U_{21}^2 + (1-\xi) U_{22}^2\}^{1/2}} \quad \text{and} \quad Z_2 = \frac{U_{22}^2}{U_{21}^2}. \]
Then \eqref{eq:bf-aeqn3} can be rewritten as 
\begin{equation}
\label{eq:bf-aeqn4}
\frac{\Ybar-\psi}{f(S_1,S_2)} = Z_1(\xi) \quad \text{and} \quad \frac{n_1 S_2^2}{n_2 S_1^2} = Z_2 \, \frac{1-\xi}{\xi}. 
\end{equation}
\citet{hsu1938} shows that $Z_1^\star \sim {\sf t}_{n_1 \wedge n_2-1}$ is stochastically fatter than $Z_1(\xi)$ for all $\xi$.  Let $G$ denote the distribution function of $Z_1^\star$, and let $W \sim \unif(0,1)$.  Then the corresponding version of \eqref{eq:regular3a} is 
\[ \frac{\Ybar-\psi}{f(S_1,S_2)} = G^{-1}(W). \]
We can get a uniformly valid predictive random set for $Z_1(\xi)$ by choosing an ordinarily valid predictive random set for $W$.  If we use the default predictive random set in \eqref{eq:default.prs} for $W$, then our generalized marginal IM plausibility intervals for $\psi$ match up with the confidence intervals in \citet{hsu1938} and \citet{scheffe1970}.  Also, the validity result from Theorem~\ref{thm:unif.valid.im} gives an alternative proof of the conservative coverage properties of the Hsu--Scheff\'e interval.  Finally, when both $n_1$ and $n_2$ are large, the bound $Z_1^\star$ on $Z_1(\xi)$ is tight.  So, at least for large samples, the generalized marginal IM is efficient.

\subsubsection{Gamma mean}
\label{SS:gamma.mean}

Let $X_1,\ldots,X_n$ be observations from a gamma distribution with unknown shape parameter $\alpha > 0$ and unknown mean $\psi$.  Here the goal is inference on the mean.  This problem has received attention, for it involves an exponential family model where a ratio of canonical parameters is of interest and there is no simple way to do marginalization.  Likelihood-based solutions are presented in \citet{fraser.reid.wong.1997}, and \citet{kulkarni.powar.2010} take a different approach.  Here we present a generalized marginal IM solution.  

The gamma model admits a two-dimensional minimal sufficient statistic for $\theta=(\psi,\alpha)$, which we will take as 
\[ T_1 = \log\Bigl( \frac1n \sum_{i=1}^n X_i \Bigr) \quad \text{and} \quad T_2 = \frac1n \sum_{i=1}^n \log X_i. \]
The most natural choice of association between data, $\theta$, and auxiliary variables is 
\[ T_1 = \log\Bigl( \frac1n \sum_{i=1}^n U_i' \Bigr) + \log\frac{\psi}{\alpha} \quad \text{and} \quad T_2 = \frac1n \sum_{i=1}^n \log U_i' + \log\frac{\psi}{\alpha}, \]
where $U_1',\ldots,U_n'$ are iid gamma random variables with both shape and mean equal to $\alpha$.  This association can be simplified by writing 
\[ T_1 - \log \psi = U_1(\alpha) \quad \text{and} \quad T_2 - \log \psi = U_2(\alpha), \]
where $U_1(\alpha)$ and $U_2(\alpha)$ are distributed as $\log( n^{-1}\sum_{i=1}^n U_i'/\alpha)$ and $n^{-1}\sum_{i=1}^n \log (U_i'/\alpha)$, respectively.   

Next, define $V_1 = U_1(\alpha)$ and $V_2 = U_1(\alpha)-U_2(\alpha)$.  For notational simplicity, we have omitted the dependence of $(V_1,V_2)$ on $\alpha$.  It is easy to check that $n\alpha e^{V_1}$ has a gamma distribution with shape parameter $n\alpha$; write $F_{n\alpha}$ for the corresponding gamma distribution function.  Let $\kappa(V_2)$ be an estimator of $\alpha$ based on $V_2$ alone.  This estimator could be a moment estimator or perhaps a maximum likelihood estimator based on the the marginal distribution of $V_2$; we give explicit estimators later.  Suppose that this estimator is consistent, i.e., $\kappa(V_2) \to \alpha$ in probability, with respect to the marginal distribution of $V_2$, as $n \to \infty$.  Then, as $n \to \infty$,  
\[ \Phi^{-1}\bigl( F_{n \kappa(V_2)}(n\kappa(V_2)e^{V_1}) \bigr) \to {\sf N}(0,1), \quad \text{in distribution}, \]
under the joint distribution of $V_1$ and $V_2$, for any $\alpha > 0$.  

The limit distribution result above is the beginning, rather than the end, of our analysis.  Indeed, consider the new association
\[ \Phi^{-1}\bigl( F_{n\kappa(T_1-T_2)}(n\kappa(T_1-T_2) e^{T_1 - \log \psi}) \bigr) = \Phi^{-1}\bigl( F_{n \kappa(V_2)}(n\kappa(V_2)e^{V_1}) \bigr). \]
Let $Z_1(\alpha)$ be a random variable with the same distribution as the quantity on the right-hand side.  We know that $Z_1(\alpha) \to {\sf N}(0,1)$, in distribution, as $n \to \infty$, for all $\alpha$. 

Take $\kappa(V_2)$ to be a moment-based estimator defined as follows.  The expectation of $V_2$ is equal to $g(n\alpha) - g(\alpha) - \log(n)$, where $g$ is the digamma function.  Then $\kappa(v_2)$ is defined as the solution for $\alpha$ in the equation $v_2 = g(n\alpha) - g(\alpha) - \log(n)$.  Similar equations appear in \citet{jensen1986} and \citet{fraser.reid.wong.1997} in a likelihood context.  Consistency of this moment estimator, as $n \to \infty$, is straightforward.  Moreover, for this $\kappa(V_2)$, we can derive limiting distributions for $Z_1(\alpha)$ as $\alpha \to \{0,\infty\}$.  Indeed, \citet{jensen1986} shows that $2n\alpha V_2$ converges in distribution to ${\sf ChiSq}(2n-2)$ and ${\sf ChiSq}(n-1)$ when $\alpha \to 0$ and $\alpha \to \infty$, respectively.  So, from this and the asymptotic approximation $g(x) = \log x - 1/(2x)$ for large $x$, it follows that $Z_1(\alpha) \to {\sf t}_{n-1}$, in distribution, as $\alpha \to \infty$.  A corresponding limit distribution as $\alpha \to 0$ is available, but we will not need this.  

We have a version of \eqref{eq:regular2a} given by 
\[ \Phi^{-1}\bigl( F_{n\kappa(T_1-T_2)}(n\kappa(T_1-T_2) e^{T_1 - \log \psi}) \bigr) = Z_1(\alpha), \]
and the goal is to find a uniformly valid predictive random set for $Z_1(\alpha)$.  We will proceed by finding a random variable $Z_1^\star$ that is stochastically fatter than $Z_1(\alpha)$ for all $\alpha$.  Unfortunately, the limit distribution ${\sf t}_{n-1}$ is not a suitable bound.  But it turns out that a relatively simple adjustment will do the trick.  First, take $\hat\nu$ as the projection of $2n\kappa(V_2)V_2$ onto $[n-1,2(n-1)]$.  Then, we define $\kappa^\star(V_2) = m(\hat\nu)/2nV_2$, where $m(\hat\nu)$ is the median of the ${\sf ChiSq}(\hat\nu)$ distribution.  With this adjusted estimator, we claim that the $Z_1^\star$ is stochastically fatter than $Z_1(\alpha)$ for all $\alpha$, where 
\begin{equation}
\label{eq:zstar}
Z_1^\star \sim 0.5 \, {\sf t}_{n-1}^+ + 0.5 \, (c_n \, {\sf t}_{n-1})^-.
\end{equation}
Here ${\sf t}_{n-1}^+$ is the positive half ${\sf t}$ distribution, and $(c_n \, {\sf t}_{n-1})^-$ is the negative half scaled-${\sf t}$ distribution.  The scaling factor $c_n$ on the negative side is  given by
\[ c_n = \{ 2(n-1) / m(2n-2) \}^{1/2}. \]  
Theoretical results justifying the claimed bound are available but, for brevity, we will only show a picture.  Figure~\ref{fig:gamma} shows the distribution of $Z_1(\alpha)$, based on the adjusted $\kappa^\star(V_2)$, for $n=2$ and $n=5$, for a range of $\alpha$, along with the distribution function corresponding to the mixture in \eqref{eq:zstar}.  The claim that $Z_1^\star$ is stochastically fatter than $Z_1(\alpha)$ is clear from the picture; in fact, the bound is quite tight for $n$ as small as 5.  

\begin{figure}
\begin{center}
\subfigure[$n=2$]{\scalebox{0.5}{\includegraphics{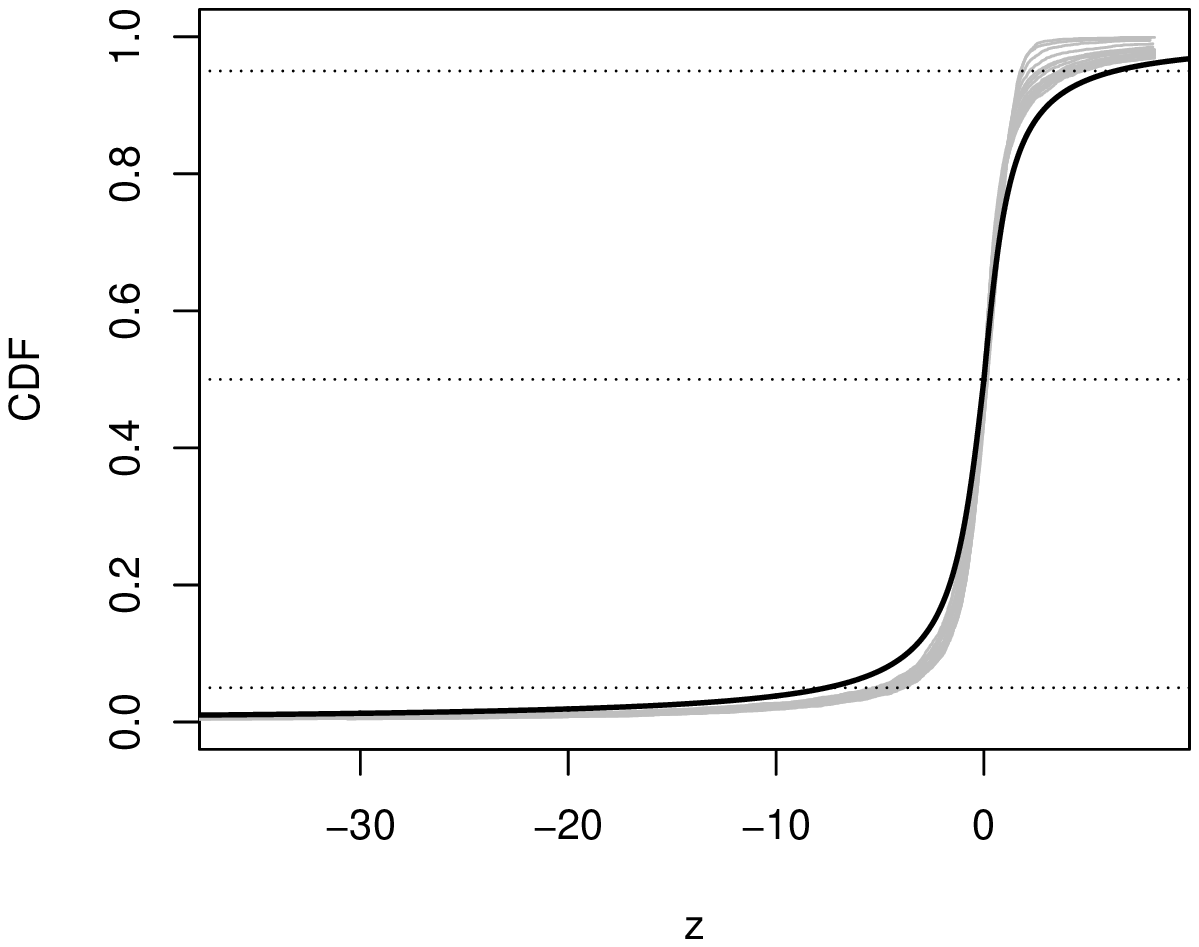}}}
\subfigure[$n=5$]{\scalebox{0.5}{\includegraphics{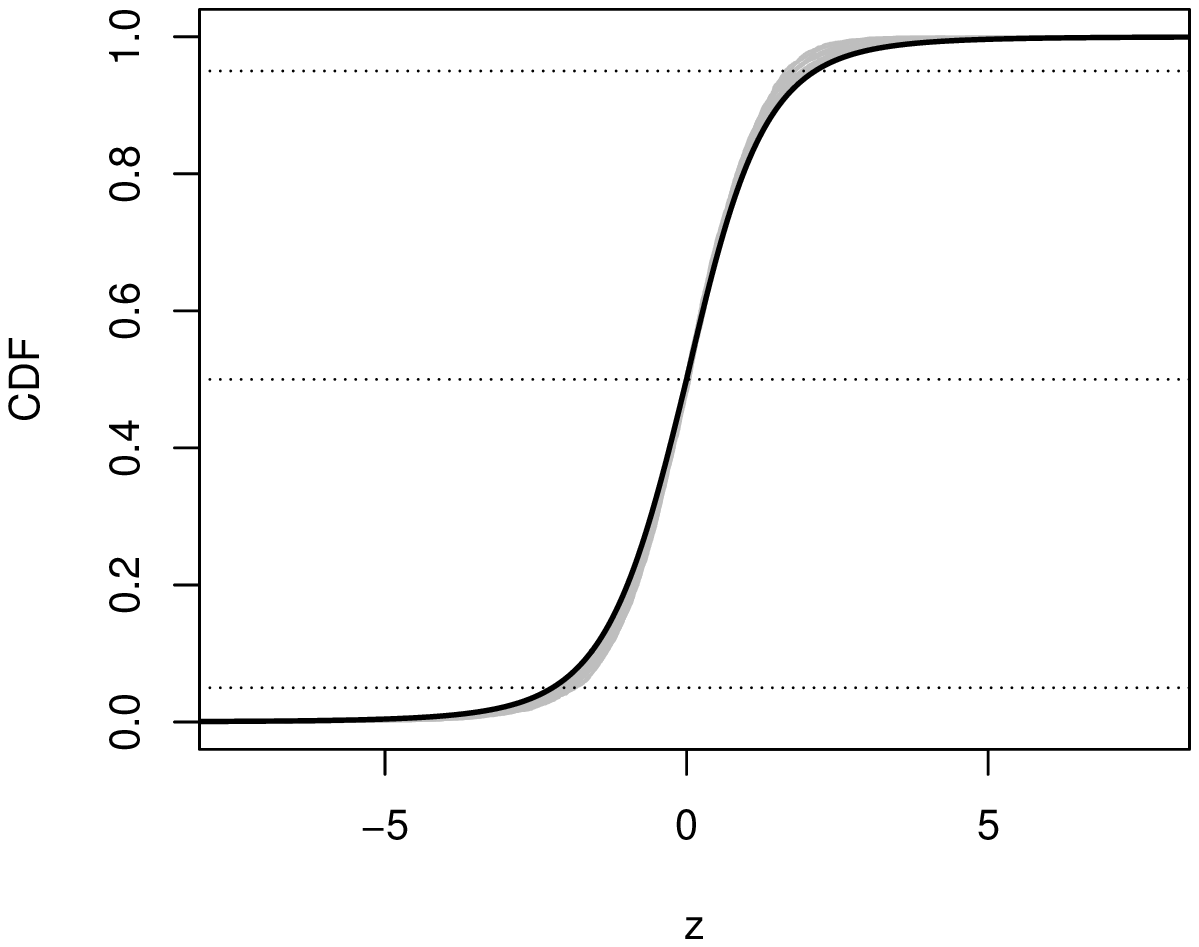}}}
\end{center}
\caption{Distribution functions of $Z_1(\alpha)$, based on the adjusted $\kappa^\star(V_2)$, over a range of $\alpha$.  Heavy line is the bound \eqref{eq:zstar}.}
\label{fig:gamma}
\end{figure}

If $G$ is the distribution function for the mixture in \eqref{eq:zstar}, and $W \sim \unif(0,1)$, then we can get a uniformly valid predictive random set for $Z_1(\alpha)$ by choosing a ordinarily valid predictive random set for $W$, such as the default \eqref{eq:default.prs}.  
From here, constructing the generalized marginal IM for $\psi$ is straightforward.  

For illustration, we consider data on survival times of rats exposed to radiation given in \citet{fraser.reid.wong.1997}, modeled as an independent gamma sample with mean $\psi$.  The 95\% generalized marginal IM plausibility interval for $\psi$ is $(96.9, 134.4)$.  Second- and third-order likelihood-based 95\% confidence interval for $\psi$ are $(97.0, 134.7)$ and $(97.2, 134.2)$, respectively.  The third-order likelihood interval is the shortest, but the plausibility interval has guaranteed coverage for all $n$, so a direct comparison is difficult.  In simulations (not shown), for $n=2$, the generalized marginal IM is valid, but conservative; see Figure~\ref{fig:gamma}(a).  For larger $n$, the likelihood and IM methods are comparable.  

\section{Discussion}
\label{S:discuss}

This paper focused on the problem of inference in the presence of nuisance parameters, and proposed a new strategy for auxiliary variable dimension reduction within the IM framework.  This reduction in dimension generally improves inference.  The regular versus non-regular classification introduced here shows which problems admit exact and efficient marginalization.  In the regular case, this marginalization can be accomplished efficiently using the techniques describe herein.  For non-regular problems, we propose a strategy based on uniformly valid predictive random sets, and one technique to construct these random sets using stochastic bounds.  While this simple strategy  maintains validity of the marginal IM, these can be conservative when $n$ is small.  Therefore, work is needed to develop more efficient marginalization strategies in non-regular problems.  

The dimension reduction considerations here are of critical importance for all statisticians working on high-dimensional problems.  In our present context, we have information that only a component of the parameter vector is of interest, and so we should use this information to reduce the dimension of the problem.  More generally, in particular in high-dimensional applications, there is information available about $\theta$, such as sparsity, and the goal is to incorporate this information and improve efficiency.  There are a variety of ways one can accomplish this, but all amount to a kind of dimension reduction.  So it is possible that the dimension reduction considerations here can help shed light on this important issue in modern statistical problems.

\section*{Acknowledgement}

The authors thank the Editor, Associate Editor, and three anonymous referees for their critical comments and suggestions, and Dr.~Jing-Shiang Hwang for helpful discussion on an earlier draft of this paper.  This work is partially supported by the U.S.~National Science Foundation, grants DMS--1007678, DMS--1208833, and DMS--1208841.



\bibliographystyle{apalike}
\bibliography{/Users/rgmartin/Dropbox/Research/mybib}

\end{document}